\documentclass[12pt,a4paper]{article}
\usepackage{amsmath, amsfonts, amsthm, latexsym, amscd}
\usepackage{amsmath}
\usepackage{xcolor,graphicx,hyperref}
\usepackage{amsfonts}
\usepackage{amssymb}
\usepackage[all]{xy}

\usepackage[norefs,nocites]{refcheck}

\allowdisplaybreaks[4]

\newtheorem{theorem}{Theorem}[section]
\newtheorem{proposition}[theorem]{Proposition}
\newtheorem{corollary}[theorem]{Corollary}
\newtheorem{example}[theorem]{Example}
\newtheorem{examples}[theorem]{Examples}

\newtheorem{lemma}[theorem]{Lemma}
\newtheorem{final remark}[theorem]{Final Remark}
\newtheorem{definition}[theorem]{Definition}

\begin{document}

\title{Maximal ideals of generalized summing linear operators}
\author{Geraldo Botelho\thanks{Supported by CNPq Grant 304262/2018-8 and Fapemig Grant PPM-00450-17.}\,, Jamilson R. Campos and  Lucas Nascimento\thanks{Supported by a CAPES scholarship. \newline 2020 Mathematics Subject Classification: 47L20, 46J20, 46B45, 47B37, 46B10.\newline Keywords: Tensor norms, operator ideals, summing operators, sequence spaces.}}
\date{}
\maketitle

\begin{abstract} We prove when a Banach ideal of linear operators defined, or characterized, by the transformation of vector-valued sequences is maximal. Known results are recovered as particular cases and new information is obtained. To accomplish this task we study a tensor quasi-norm determined by the underlying sequence classes. The duality theory for these tensor quasi-norms is also developed.
\end{abstract}

\section{Introduction}
The theory of operator ideals is central in modern mathematical analysis (see \cite{handbook}, \cite[6.3]{history}) and, in this context, maximal ideals play a key role. For recent developments on maximal operator ideals, see, e.g, \cite{samuel, kim, J. A. Lopez Molina.tres, turcovillafane}. A number of important operator ideals are defined, or characterized, by the transformation of vector-valued sequences; and some of these ideals are known to be maximal. A unifying approach to this kind of operators ideals was proposed in \cite{G. Botelho} using the concept of {\it sequence classes}. For sequence classes $X$ and $Y$, a linear operator $T \colon E \longrightarrow F$ is $(X;Y)$-summing if $T((x_j))_{j=1}^\infty \in Y(F)$ whenever $(x_j)_{j=1}^\infty \in X(E)$. The Banach operator ideal of such operators is denoted by ${\cal L}_{X;Y}$.  This approach has proved to be quite fruitful, see \cite{achour2018, complut, Jamilson.dual, espaco.mid, G. Botelho and D. Freitas, raquel, J.R.Campos.J.Santos, J. Ribeiro and F. Santos, J. Ribeiro and F. Santos.dois}. The purpose of this paper is to study the maximality of these Banach operator ideals. Generalizing some well known cases, we find conditions on the sequences classes $X$ and $Y$ so that the Banach operator ideal ${\cal L}_{X;Y}$ is maximal.

Following the long tradition of the interplay between operator ideals and tensor norms, which comes from Grothendieck's seminal works and stands to the day (see recent developments in \cite{achour2018, D.Achour, sheldon, maite, maite2020, kim2020,  J. A. Lopez Molina, J. A. Lopez Molina.tres, lopezmolina2019, miguel}), we prove our main results defining, developing and applying a tensor quasi-norm $\alpha_{X;Y}$ determined by the sequence classes $X$ and $Y$.  The tensor quasi-norms $\alpha_{X;Y}$ can be regarded as generalizations of the classical Chevet-Saphar tensor norms (see \cite{A.Defant, R. Ryan}).

In Section 2 we define a tensor quasi-norm associated to the underlying sequences classes and apply it to give conditions so that the corresponding ideal of summing operators is maximal. Known results are recovered as particular instances and new concrete information is obtained. The duality theory of the tensor quasi-norm $\alpha_{X;Y}$ associated to the sequence classes $X$ and $Y$ is developed in Section 3. For Banach spaces $E$ and $F$ we describe the continuous linear functionals on $E \otimes_{\alpha_{X;Y}} F$ as linear operators from $E$ to $F^*$ and as continuous bilinear forms on $E \times F$. As a byproduct we show when the tensor quasi-norm $\alpha_{X;Y}$ satisfies a condition that is equivalent to maximality in the case of operator ideals associated to finitely generated tensor norms.

For operator ideals we refer to \cite{A.Defant, A.Pietsch}, for the interplay between tensor norms and operator ideals to \cite{A.Defant, R. Ryan}, for the theory of absolutely summing operators to \cite{J.Diestel} and for quasi-norms and quasi-normed spaces to \cite{N. J. Kalton 2}.

Banach spaces over $\mathbb{K} = \mathbb{R}$ or $\mathbb{C}$ shall be denoted by $E$ and $F$. The closed unit ball of $E$ is denoted by $B_E$ and its topological dual by $E^*$.  The symbol $E \stackrel{1}{\hookrightarrow} F$ means that $E$ is a linear subspace of $F$ and $\|x\|_F \leq \|x\|_E$ for every $x \in E$; and $E\stackrel{1}{=}F$ means that $E$ is isometrically isomorphic to $F$.

By $L(E;F)$ we denote the space of all linear operators from $E$ to $F$ and by $\mathcal{L}(E;F)$ the Banach space of all continuous linear operators $T:E\longrightarrow F$ endowed with the usual sup norm. The same notation will be used if $E$ and $F$ are quasi-normed spaces.

For $x \in E$ and $j \in \mathbb{N}$, the symbol $x\cdot e_j$ denotes the sequence $(0,\ldots, 0,x,0, 0,\ldots ) \in E^\mathbb{N}$, where $x$ is placed at the $j$-th coordinate. The symbol $(x_{j})_{j=1}^{n}$, where $x_1, \ldots, x_n \in E$, stands for the sequence $(x_{1},x_{2},\ldots,x_{n},0,0,\ldots) \in E^\mathbb{N}$.

According to \cite{G. Botelho}, a {\it sequence class} is a rule $X$ that assigns to each Banach space $E$ a Banach space $X(E)$ of $E$-valued sequences, that is $X(E)$ is a vector subspace of $E^{\mathbb{N}}$ with the coordinatewise operations, such that:\\
(i) $c_{00}(E) \subseteq X(E) \stackrel{1}{\hookrightarrow}  \ell_\infty(E)$ for every Banach space $E$,\\
(ii) $\|x \cdot e_j\|_{X(E)}= \|x\|_E$  for every Banach space $E$, every $x \in E$ and every $j \in \mathbb{N}$.

To avoid ambiguity, we shall eventually denote the sequence class $X$ by $X(\cdot)$.

Given sequences classes $X$ and $Y$, we say that an operator $T \in {\cal L}(E;F)$ is {\it $(X;Y)$-summing} if $T((x_j))_{j=1}^\infty \in Y(F)$ whenever $(x_j)_{j=1}^\infty \in X(E)$. In this case, the induced linear operator
 $$\widehat{T} \colon X(E) \longrightarrow Y(F)~,~\widehat{T}\left( (x_j)_{j=1}^\infty \right) = \left(T(x_j)\right)_{j=1}^\infty ,$$
is continuous and
$$\|T \|_{X;Y} : = \|\widehat{T}\| $$
is a norm that makes the space ${\cal L}_{X;Y}(E;F)$ of $(X;Y)$-summing operators a Banach space. Whenever we refer to ${\cal L}_{X;Y}(E;F)$ we assume that it is endowed with the norm $\|\cdot \|_{X;Y}$.

A sequence class $X$ is {\it linearly stable} if, regardless of the Banach spaces $E$ and $F$, every operator $T \in {\cal L}(E;F)$ is $(X;X)$-summming and $\|T\|_{X;X} = \|T\|$, that is, $\mathcal{L}_{X;X}(E;F)\stackrel{1}{=} \mathcal{L}(E;F)$.

If the sequence classes $X$ and $Y$ are linearly stable and $X(\mathbb{K}) \stackrel{1}{\hookrightarrow} Y(\mathbb{K})$, then $\mathcal{L}_{X;Y}$ is a Banach operator ideal \cite[Theorem 3.6]{G. Botelho}.

\begin{example}\label{exsec}\rm Let $p \geq 1$. The following are well known linearly stable sequence classes, endowed with their usual norms: \\
$\bullet$ The class $E \mapsto c_0(E)$ of norm null sequences.\\
$\bullet$ The class $E \mapsto \ell_\infty(E)$ of bounded sequences.\\
$\bullet$ The class $E \mapsto \ell_p(E)$ of absolutely $p$-summable sequences.\\
$\bullet$ The class $E \mapsto \ell_p^w(E)$ of weakly $p$-summable sequences.\\
$\bullet$ The class $E \mapsto \ell_p\langle E \rangle$ of Cohen strongly $p$-summable sequences.\\
$\bullet$ The class $E \mapsto {\rm Rad}(E)$ of almost unconditionally summable sequences.

Consider
$${\rm RAD}(E) : = \left\{(x_{j})_{j=1}^{\infty} \in E^{\mathbb{N}} : \|(x_{j})_{j=1}^{\infty}\|_{{\rm RAD}(E)} := \sup\limits_{k} \|(x_{j})_{j=1}^{k}\|_{{\rm Rad}(E)} < \infty\right\}$$
(see \cite{J.Diestel, tarieladze}),
$$\ell_p^{\rm mid}(E) :=  \left\{(x_{j})_{j=1}^{\infty} \in E^{\mathbb{N}} : \|(x_{j})_{j=1}^{\infty}\|_{{\rm mid},p} := \sup\limits_{(\varphi_n)_{n=1}^\infty \in B_{\ell_p^w(E^*)}} \left(\sum_{j,n = 1}^\infty |\varphi_n(x_j)|^p \right)^{1/p} < \infty\right\},$$
and the closed subspace $\ell_p^u(E)$ of $\ell_p^w(E)$ formed by unconditionally $p$-summable sequences, that is
$$\ell_p^u(E) = \left\{(x_j)_{j=1}^\infty \in \ell_p^w(E) : \lim\limits_k \|(x_j)_{j=k}^\infty\|_{w,p} = 0\right\} $$
(see  \cite{A.Defant}).  Then the correspondences $E \mapsto {\rm RAD}(E)$, $E \mapsto \ell_p^{\rm mid}(E)$ and $E \mapsto \ell_p^u(E)$ are also linearly stable sequences classes.
\end{example}

The dual of a sequence class $X$ was introduced in \cite{Jamilson.dual} in the following fashion:
\begin{equation*}
	X^{\rm dual}(E) = \left\{(x_j)_{j=1}^\infty\ \mathrm{in\ } E: \sum_{j=1}^\infty \varphi_j(x_j)\ \mathrm{converges\ } \text{for every}\ (\varphi_j)_{j=1}^\infty\ \mathrm{in\ } X(E^*)\right\}.
\end{equation*}

A sequence class $X$ is {\it spherically complete} if $(\lambda_jx_j)_{j=1}^\infty \in X(E)$ and $\|(\lambda_jx_j)_{j=1}^\infty \|_{X(E)} = \|(x_j)_{j=1}^\infty\|_{X(E)}$ whenever $(x_j)_{j=1}^\infty \in X(E)$ and $(\lambda_j)_{j=1}^\infty \in \mathbb{K}^\mathbb{N}$ is such that $|\lambda_j| = 1$ for every $j$.

For example, the sequence classes $c_0(\cdot), \ell_\infty(\cdot), {\rm RAD}(\cdot), \ell_p(\cdot), \ell_p^w(\cdot), \ell_p\langle \cdot \rangle, \ell_p^u(\cdot), \ell_p^{\rm mid}(\cdot), 1 \leq p < \infty$, are spherically complete.

Let $X$ be a linearly stable and spherically complete sequence class. In \cite{Jamilson.dual} it is proved that the expression
\begin{equation*}
\left\|(x_{j})_{j=1}^{\infty} \right\|_{X^{\rm dual}}:= \sup_{(\varphi_{j})_{j=1}^{\infty}\in B_{X(E^*)} }\sum_{j=1}^{\infty}\left|\varphi_{j}(x_{j}) \right| = \sup_{(\varphi_{j})_{j=1}^{\infty}\in B_{X(E^*)} }\left|\sum_{j=1}^{\infty}\varphi_{j}(x_{j}) \right|
\end{equation*}
makes $X^{\rm dual}(E)$ a Banach space and $X^{\rm dual}(\cdot)$ a linearly stable spherically complete sequence class. Conditions on $X$ so that $X^{\rm dual}(E^*)$ is canonically isometrically isomorphic to $X(E)^*$ are also given in \cite{Jamilson.dual}.

For example, for $1 \leq p \leq \infty$ and $\frac{1}{p} + \frac{1}{p^*} = 1$, $(\ell_p^w)^{\rm dual} = \ell_{p^*}\langle \cdot \rangle$ and $(\ell_p)^{\rm dual} = \ell_{p^*}( \cdot )$ (the case $p = \infty$ of this last equality is somewhat surprising).

\section{Maximal ideals}
The purpose of this section is to find conditions on $X$ and $Y$ so that ${\cal L}_{X;Y}$ is a maximal Banach operator ideal. Oddly enough, we begin by giving plenty of counterexamples.
\begin{example}\label{exnew}\rm For $1 \leq p < \infty$, by ${\cal C}_p$ we denote the ideal of $p$-converging operators (see, e.g., \cite{chen}), that is, the operators that send weakly $p$-summable sequences to norm null sequences. In our notation, ${\cal C}_p = {\cal L}_{\ell_p^w(\cdot); c_0(\cdot)}$. The case $p = 1$ recovers the ideal of unconditionally summing operators from \cite[1.7.1]{A.Pietsch}. All these ideals are not maximal, see, e.g., \cite[Theorem 2.7]{chen}. In the same fashion, the classical ideal ${\cal CC} : = {\cal L}_{c_o^w(\cdot); c_0(\cdot)}$ of completely continuous operators is not maximal.
\end{example}
Since $c_0$ is not a dual space, there is no sequence class $Y$ such that $Y^{\rm dual} = c_0(\cdot)$. Our guess is that this is the fact behind the non maximality of the ideals in Examples \ref{exnew}. So, the search for maximality should be restricted to operator ideals of  the form ${\cal L}_{X;Y^{\rm dual}}$. As announced, we shall do that by considering tensor quasi-norms.

Of course we need ${\cal L}_{X;Y^{\rm dual}}$ to be a Banach operator ideal. So, according to the linear case of \cite[Theorem 3.6]{G. Botelho}, whenever we refer to ${\cal L}_{X;Y^{\rm dual}}$ in this section we assume that the sequence classes $X$ and $Y$ are linearly stable, $Y$ is spherically complete and $X(\mathbb{K}) \stackrel{1}{\hookrightarrow} Y^{\rm dual}(\mathbb{K})$. The linear cases of the characterizations proved in \cite[Proposition 2.4]{G. Botelho} shall be used making no explicit reference.

A careful look at the definition of reasonable tensor norm and at the proof of \cite[Proposition 6.1]{R. Ryan} makes the following definition quite natural.

\begin{definition}\rm Let $\varepsilon$ be the injective tensor norm. For Banach spaces $E$ and $F$, a quasi-norm $\alpha$ on $E\otimes F$ is said to be {\it reasonable} if $\varepsilon \leq \alpha$ and $ \alpha(x \otimes y) \le \|x\|\cdot\|y\|$ for all $x \in E$ and $y \in F$.
\end{definition}

Let $X$ and $Y$ be sequence classes. For Banach spaces $E$ and $F$, consider the map $\alpha_{X,Y} \colon E\otimes F\longrightarrow \mathbb{R}$ given by
$$\alpha_{X,Y}^{}(u)= \inf\left\lbrace \left\|(x_{j})_{j=1}^{n} \right\|_{X(E)} \cdot  \left\|(y_{j})_{j=1}^{n} \right\|_{Y(F)} : u=\sum_{j=1}^{n}x_{j}\otimes y_{j} \right\rbrace. $$

Only one condition on the sequences classes $X$ and $Y$ is needed for $\alpha_{X,Y}$ to be a reasonable quasi-tensor norm. The notion we define now is quite weaker than the related ones that can be found in the literature (see, e.g., \cite{Botelhojlucas, argentinos}).

A sequence class $X$ is \emph{monotone} if for every Banach space $E$ and all $m,n \in \mathbb{N}$ and $x_1, \ldots, x_n \in E$, the following holds: 
$$\|(\underbrace{0,0, \ldots, 0}_{m \rm{\,times}}, x_1, \ldots, x_n, 0,0, \ldots)\|_{X(E)} = \|(x_1, \ldots, x_n, 0,0,\ldots)\|_{X(E)}. $$

All sequence classes in Example \ref{exsec} are monotone.

\begin{proposition}\label{razoavel}
If $X$ and $Y$ are monotone sequence classes and $\varepsilon\leq \alpha_{X,Y}$, then $\alpha_{X,Y}^{}$ is a reasonable quasi-norm on $E\otimes F$ for all $E$ and $F$.
\end{proposition}

\begin{proof}
Let us show that, for all Banach spaces $E$ and $F$ and all  $u_{1},u_{2}\in E\otimes F$ $$\alpha_{X,Y}^{}(u_{1}+u_{2})\leq 2\left( \alpha_{X,Y}^ {}(u_{1}) + \alpha_{X,Y}^{}(u_{2})\right).$$
Given $\eta>0$, choose representations $u_{i}=\sum\limits_{j=1}^{n}x_{ij}\otimes y_{ij}$ such that 	$$\left\|(x_{ij})_{j=1}^{n} \right\|_{X(E)}\leq \alpha_{X,Y}^{}(u_{i} ) + \eta \  \text{ and }  \ \left\|(y_{ij})_{j=1}^{n} \right\|_{Y(F)}\leq 1, \, i=1,2.$$
  So, $\sum\limits_{j=1}^{n}x_{1j}\otimes y_{1j}+ \sum\limits_{j=1}^{n}x_{2j} \otimes y_{2j}$ is a representation of $u_{1}+ u_{2}$.
  Using that $X$ and $Y$ are monotone, we get $\alpha_{X,Y}^{}(u_{1}+u_{2}) \leq$
\begin{align*}
	&\leq\|(x_{11}, \ldots, x_{1n}, x_{21}, \ldots, x_{2n},0,0,\ldots) \|_{X(E)}\cdot \| (y_{11}, \ldots, y_{1n}, y_{21}, \ldots, y_{2n},0,0,\ldots) \|_{Y(F)}\\
& \leq \left(\|(x_{11}, \ldots, x_{1n}, 0,0,\ldots) \|_{X(E)}+ \|(0, \ldots, 0, x_{21}, \ldots, x_{2n},0,0,\ldots) \|_{X(E)}\right)\cdot \\
& ~~~\cdot \left(\|(y_{11}, \ldots, y_{1n}, 0,0,\ldots) \|_{Y(F)}+ \|(0, \ldots, 0, y_{21}, \ldots, y_{2n},0,0,\ldots) \|_{Y(F)}\right)\\
& = \left(\|(x_{11}, \ldots, x_{1n}, 0,0,\ldots) \|_{X(E)}+ \|(x_{21}, \ldots, x_{2n},0,0,\ldots) \|_{X(E)}\right)\cdot \\
& ~~~\cdot \left(\|(y_{11}, \ldots, y_{1n}, 0,0,\ldots) \|_{Y(F)}+ \|(y_{21}, \ldots, y_{2n},0,0,\ldots) \|_{Y(F)}\right)\\
	&\leq  2\left(\alpha_{X,Y}^{}(u_{1}) + \alpha_{X,Y}^{}(u_{2}) + 2\eta \right).
	\end{align*}
The desired inequality follows by making $\eta \longrightarrow 0^+$.	

The other facts either follow from the definition of sequence class or are straightforward.
\end{proof}

Given a (non necessarily continuous) linear operator $T \colon E \longrightarrow F$, it is clear that  $$A_{T} \colon E \times F^*\longrightarrow \mathbb{K}~,~A_{T}(x,\psi)=\psi (T(x)),$$
is a (non necessarily continuous) bilinear form. Calling $\varphi_T$ the linearization of $A_T$, we have a  linear functional $\varphi_{T}\colon E\otimes F^*\longrightarrow \mathbb{K}$ satisfying
$$\varphi_{T}\left(\sum_{j=1}^{n}x_{j}\otimes \psi_{j} \right)= \sum_{j=1}^{n}\psi_{j}(T(x_{j}))$$
for every $\sum\limits_{j=1}^{n}x_{j}\otimes \psi_{j} \in E \otimes F^*$. Of course, the map $T \mapsto \varphi_T$ is linear.

To proceed we need to recall one more definition from \cite{G. Botelho}. A sequence class $X$ is said to be:\\
$\bullet$ \emph{Finitely determined} if for every sequence $(x_j)_{j=1}^\infty \in E^{\mathbb{N}}$, $(x_j)_{j=1}^\infty \in X(E)$ if and only if $\displaystyle\sup_k \left\|(x_j)_{j=1}^k  \right\|_{X(E)} < +\infty$ and, in this case, $\left\|(x_j)_{j=1}^\infty  \right\|_{X(E)} = \sup_k \left\|(x_j)_{j=1}^k  \right\|_{X(E)}. $\\
$\bullet$ {\it Finitely dominated} if there is a finitely determined sequence class $Y$ such that, for every Banach space $E$, $X(E)$ is a closed subspace of $Y(E)$ and one of the following conditions holds:\\
(i) For every sequence $(x_j)_{j=1}^\infty \in Y(E)$, $(x_j)_{j=1}^\infty \in X(E)$ if and only if $\displaystyle\lim_k \|(x_j)_{j=k}^\infty\|_{Y(E)} = 0$.\\ 
(ii) For every sequence $(x_j)_{j=1}^\infty \in Y(E)$, $(x_j)_{j=1}^\infty \in X(E)$ if and only if $\displaystyle\lim_{k,l} \|(x_j)_{j=k}^l\|_{Y(E)} = 0$.

For example, for $1 \leq p < \infty$, the classes $\ell_p(\cdot), \ell_p^w(\cdot), \ell_p\langle \cdot \rangle, \ell_p^{\rm mid}(\cdot), \ell_\infty(\cdot)$ and ${\rm RAD}(\cdot)$ are finitely determined; $c_0(\cdot)$ is finitely dominated by $\ell_\infty(\cdot)$, $\ell_p^u(\cdot)$ is finitely dominated by $\ell_p^w(\cdot)$ and ${\rm Rad}(\cdot)$ is finitely dominated by ${\rm RAD}(\cdot)$. Moreover, the dual $Y^{\rm dual}$ of a sequence class $Y$ is always finitely determined, even if $Y$ is not.

\begin{lemma}\label{1lema.ideal.maximal.qn}
	Let $T\in \mathcal{L}(E;F)$ and suppose that $\alpha_{X,Y}^{}$ is a reasonable quasi-norm.\\
 {\rm	(a)} If $T\in \mathcal{L}_{X;Y^{\rm dual}}(E;F)$, then $\varphi_{T}\colon E\otimes_{\alpha_{X,Y}^{}} F^*\longrightarrow \mathbb{K}$	is continuous and $\left\|\varphi_{T} \right\| \leq \left\|T \right\|_{X;Y^{\rm dual}}$.\\
 {\rm (b)} If, in addition, $X$ and $Y$ are finitely determined or finitely dominated, then the functional $\varphi_{T} \colon E\otimes_{\alpha_{X,Y}^{}} F^*\longrightarrow \mathbb{K}$ is continuous if and only if $T\in \mathcal{L}_{X;Y^{\rm dual}}(E;F)$; and, in this case, $\left\|T \right\|_{X;Y^{\rm dual}}= \left\|\varphi_{T} \right\|$.
\end{lemma}

\begin{proof} (a)  For every $u=\sum\limits_{j=1}^{n}x_{j}\otimes \psi_{j} \in E\otimes F^*$, we have
	\begin{equation*}
	\left|\varphi_{T}\left(u \right)  \right|= \left|\sum_{j=1}^{n}\psi_{j}(T(x_{j})) \right| \leq \sum_{j=1}^{n}\left|\psi_{j}(T(x_{j})) \right|
	\leq \left\|T \right\|_{X;Y^{\rm dual}} \cdot \left\|(x_{j})_{j=1}^{n} \right\|_{X(E)} \cdot \left\|(\psi_{j})_{j=1}^{n} \right\|_{Y(F^*)},
	\end{equation*}
	where the last inequality follows from a simple manipulation of the norm of $Y^{\rm dual}$.
Taking the infimum over all representations of $u$ it follows that the functional $\varphi_{T}\colon  E\otimes_{\alpha_{X,Y}^{}} F^*\longrightarrow \mathbb{K}$ is continuous with $\left\|\varphi_{T} \right\| \leq \left\|T \right\|_{X;Y^{\rm dual}}. $\\
	(b) We prove the case that $X$ and $Y$ are finitely determined. Given $(x_j)_{j=1}^\infty \in X(E)$, for every $n \in \mathbb{N}$,
	\begin{align*}
	\left\| (T(x_{j}))_{j=1}^{n} \right\|_{Y^{\rm dual}(F)} &= \sup_{(\psi_{j})_{j=1}^{\infty}\in B_{Y(F^*)}}\left|\sum_{j=1}^{n}\psi_{j}(T(x_{j})) \right| = \sup_{(\psi_{j})_{j=1}^{\infty}\in B_{Y(F^*)}}\left|\varphi_{T}\left(\sum_{j=1}^{n}x_{j}\otimes \psi_{j} \right)  \right| \\
	&\leq \sup_{(\psi_{j})_{j=1}^{\infty}\in B_{Y(F^*)}}\left\|\varphi_{T} \right\|\cdot\alpha_{X,Y}^{}\left(\sum_{j=1}^{n}x_{j}\otimes \psi_{j} \right)\\
	&\leq \left\|\varphi_{T} \right\| \cdot \left\|(x_{j})_{j=1}^{n} \right\|_{X(E)}\cdot \sup_{(\psi_{j})_{j=1}^{\infty}\in B_{Y(F^*)}}\left\|(\psi_{j})_{j=1}^{n} \right\|_{Y(F^*)}\\
&\leq \left\|\varphi_{T} \right\| \cdot \sup_k\left\|(x_{j})_{j=1}^{n} \right\|_{X(E)}\cdot \sup_{(\psi_{j})_{j=1}^{\infty}\in B_{Y(F^*)}}\sup_k\left\|(\psi_{j})_{j=1}^{k} \right\|_{Y(F^*)}\\
	&= \left\|\varphi_{T} \right\| \cdot \left\|(x_{j})_{j=1}^{\infty} \right\|_{X(E)}\cdot \sup_{(\psi_{j})_{j=1}^{\infty}\in B_{Y(F^*)}}\left\|(\psi_{j})_{j=1}^{\infty} \right\|_{Y(F^*)}\\
	&= \left\|\varphi_{T} \right\|\cdot \left\|(x_{j})_{j=1}^{\infty} \right\|_{X(E)}.
	\end{align*}
Since $Y^{\rm dual}$ is also finitely determined, taking the supremum over $n$ we get $(T(x_{j}))_{j=1}^{\infty} \in Y^{\rm dual}(F)$ and
$$\left\| (T(x_{j}))_{j=1}^{\infty} \right\|_{Y^{\rm dual}(F)} \leq \left\|\varphi_{T} \right\|\cdot \left\|(x_{j})_{j=1}^{\infty} \right\|_{X(E)}, $$		from which it follows that $T\in \mathcal{L}_{X;Y^{\rm dual}}(E;F)$  and  $ \left\|T \right\|_{X;Y^{\rm dual}}\leq \left\|\varphi_{T} \right\|$.
\end{proof}

In the same fashion of norms on tensor products (see \cite[Section 6.1]{R. Ryan}), we say that a reasonable quasi-norm $\alpha$ is {\it a tensor quasi-norm} if:\\
$\bullet$ $\alpha$ is uniform, that is, for all Banach spaces $E_1,E_2, F_1, F_2$ and all operators $T_i \in {\cal L}(E_i,F_i)$, $i = 1,2$,
$$\|T_1 \otimes T_2 \colon E_1 \otimes_\alpha E_2 \longrightarrow F_1 \otimes_\alpha F_2\| \leq \|T_1\|\cdot \|T_2\|.$$
$\bullet$ $\alpha$ is finitely generated, that is, for all Banach spaces $E,F$ and any $u \in E \otimes F$,
$$\alpha(u; E \otimes F) = \inf\left\{\alpha(u; M \otimes N) : u \in M \otimes N, M \in {\cal F}(E),N \in {\cal F}(F) \right\}, $$
where ${\cal F}(E)$ is the collection of all finite dimensional subspaces of $E$.

Let us see that, in the environment of sequences classes, tensor quasi-norms are not rare. A sequence class $X$   is said to be {\it finitely injective} if $\|(x_j)_{j=1}^k\|_{X(E)} \leq \|(i(x_j))_{j=1}^k\|_{X(F)}$ whenever $i \colon E \longrightarrow F$ is a metric injection, $k \in \mathbb{N}$ and $x_1, \ldots, x_k \in E$. If $X$ is also linearly stable, then we actually have $\|(x_j)_{j=1}^k\|_{X(E)} = \|(i(x_j))_{j=1}^k\|_{X(F)}$.

All sequence classes listed in Example \ref{exsec}, but $\ell_p\langle \cdot \rangle$, are finitely injective.

\begin{proposition} \label{porpr}
	Let $X$ and $Y$ be sequence classes such that $\alpha_{X,Y}^{}$ is a reasonable quasi-norm. If $X$ and $Y$ are linearly stable and finitely injective, then $\alpha_{X,Y}^{}$ is a tensor quasi-norm.
\end{proposition}

\begin{proof}
Given $T_{i}\in \mathcal{L}(E_{i};F_{i})$, $i=1,2$,
$u= \sum\limits_{j=1}^{n}x_{j}\otimes y_{j}\in E_1\otimes_{\alpha_{X,Y}^{}}E_2$, the linear stability of $X$ and $Y$ gives
	\begin{align*}
	\alpha_{X,Y}^{}\left(T_{1}\otimes T_{2}(u)  \right)& =\alpha_{X,Y} \left(\sum_{j=1}^n T_1(x_j) \otimes T_2(y_j) \right) \\
&\leq \left\|(T_{1}(x_{j}))_{j=1}^{n} \right\|_{X(F_{1})}\cdot \left\|(T_{2}(y_{j}))_{j=1}^{n} \right\|_{Y(F_{2})}\\
	&\leq \left\|T_{1} \right\|\cdot \left\|T_{2} \right\| \cdot \left\|(x_{j})_{j=1}^{n} \right\|_{X(E_{1})}\cdot \left\|(y_{j})_{j=1}^{n} \right\|_{Y(E_{2})}.
	\end{align*}
 Since this holds for every representation of $u$, it follows that $\alpha_{X,Y}^{}\left(T_{1}\otimes T_{2}(u)  \right)\leq \left\|T_{1} \right\|\cdot \left\|T_{2} \right\|\cdot \alpha_{X,Y}^{}(u)$, so $\left\|T_{1}\otimes T_{2} \right\| \leq \left\| T_{1}\right\| \cdot \left\| T_{2}\right\|.$ This proves that $\alpha_{X,Y}$ is uniform.

Let $u\in E\otimes F$ be given. Given $\eta>0$, we can take a representation $u = \sum\limits_{j=1}^{n}x_{j}\otimes y_{j}$ so that $$\left\|(x_{j})_{j=1}^{n} \right\|_{X(E)}\cdot\left\|(y_{j})_{j=1}^{n} \right\|_{Y(F)} \leq \alpha_{X,Y}^{}(u; E\otimes F) + \eta. $$	Taking $M= {\rm span}\{x_{1},\ldots,x_{n}\}$ and $N= {\rm span}\{y_{1}\ldots,y_{n}\}$, we have $u\in M\otimes N$ and $\alpha_{X,Y}^{}(u;E\otimes F)\leq \alpha_{X,Y}^{}(u;M\otimes N)$ because  $\alpha_{X,Y}^{}$ is uniform. Moreover, the finite injectivity of $X$ and $Y$ yields
	\begin{align*}
	\alpha_{X,Y}^{}(u;M\otimes N)&\leq \left\|(x_{j})_{j=1}^{n} \right\|_{X(M)} \cdot \left\|(y_{j})_{j=1}^{n} \right\|_{Y(N)} \\
	&\leq\left\|(x_{j})_{j=1}^{n} \right\|_{X(E)}\cdot \left\|(y_{j})_{j=1}^{n} \right\|_{Y(F)} \leq \alpha_{X,Y}^{}(u; E\otimes F) + \eta.
	\end{align*}
It follows that $\alpha_{X,Y}^{}(u;E\otimes F)= \inf\left\lbrace \alpha_{X,Y}^{}(u;M\otimes N); u\in M\otimes N, M \in {\cal F}(E),  N \in {\cal F}(F) \right\rbrace$, proving that $\alpha_{X,Y}$ is finitely generated.
\end{proof}

Recall that a Banach operator ideal $[{\cal I}, \|\cdot\|_{\cal I}]$ is {\it maximal} (see \cite[p.\,197]{R. Ryan}) if it is the only Banach operator ideal $[{\cal J}, \|\cdot\|_{\cal J}]$ satisfying:\\
(i) ${\cal I}(E;F) \subseteq {\cal J}(E;F)$ for all Banach spaces $E$ and $F$ and $\|u\|_{\cal J} \leq \|u\|_{\cal I}$ for every $u \in {\cal I}(E;F)$, and\\
(ii) $\|u\|_{\cal J} = \|u\|_{\cal I}$ for every finite rank operator.

We denote by $\mathcal{CF}(E)$ the collection of all finite codimensional subspaces of $E$. For $M \in \mathcal{F}(E)$ we denote by $I_M \colon M \longrightarrow E$ we denote the inclusion operator and for $L \in \mathcal{CF}(F)$ we denote by $Q_L \colon F \longrightarrow F/L$ the quotient operator.

\begin{theorem}\label{maxim}
	Suppose that $\alpha_{X,Y}$ is a tensor quasi-norm and that $X$ and $Y$ are finitely determined or finitely dominated. For an operator $T\in \mathcal{L}(E;F)$, $T\in \mathcal{L}_{X;Y^{\rm dual}}(E;F)$ if and only if
$$s: = \sup\left\{\left\| Q_{L}\circ T \circ I_{M}\right\|_{X;Y^{\rm dual}} : (M,L)\in \mathcal{F}(E)\times \mathcal{CF}(F)\right\} < \infty,$$
and, in this case, $\left\|T \right\|_{X;Y^{\rm dual}} = s$.
In particular, the Banach operator ideal $\mathcal{L}_{X;Y^{\rm dual}}$ is maximal.
\end{theorem}

\begin{proof} Suppose that $T\in \mathcal{L}_{X;Y^{\rm dual}}(E;F)$. For being a finite rank operator, each $Q_{L}\circ T \circ I_{M}$ belongs to  $\mathcal{L}_{X;Y^{\rm dual}}(M;F/L)$.  The ideal inequality of the norm of $\mathcal{L}_{X;Y^{\rm dual}}$ gives 
\[\left\|Q_{L}\circ T \circ I_{M} \right\|_{X;Y^{\rm dual}}\leq \left\| Q_{L}\right\|\cdot \left\| T\right\|_{X;Y^{\rm dual}}\cdot \left\| I_{M}\right\| =\left\| T\right\|_{X;Y^{\rm dual}},\] which proves that $s\leq \left\|T \right\|_{X;Y^{\rm dual}} < \infty.$
	
	Conversely, suppose that $s < \infty$. Let $u\in E\otimes F^*$ and $\eta>0$ be given. As $\alpha_{X,Y}^{}$ is finitely generated (Proposition \ref{porpr}), there are $M\in \mathcal{F}(E)$, $N\in \mathcal{F}(F^*)$ and a representation $u=\sum\limits_{j=1}^{n}x_{j}\otimes y_{j}^*\in M\otimes N$ such that \[\alpha_{X,Y}^{}\left(u; M\otimes N \right)\leq (1+\eta)\alpha_{X,Y}^{}\left(u ; E\otimes F^* \right).\]
	Let $L\in \mathcal{CF}(F)$ be such that $\left(F/L \right)^* \stackrel{1}{=}N$ by means of the operator $Q_{L}^*\colon  \left(F/L \right)^* \longrightarrow N$. Choose functionals $\psi_{j}\in \left( F/L\right)^* $ such that $Q_{L}^*(\psi_{j})= y_{j}^*, j= 1,\ldots,n$. In the chain
$$M \otimes_{\alpha_{X;Y}}  N \stackrel{Id_M \otimes (Q_L^*)^{-1}}{\xrightarrow{\hspace*{2cm}}} M \otimes_{\alpha_{X;Y}}(F/L)^* \stackrel{\varphi_{Q_L \circ T \circ I_M}}{\xrightarrow{\hspace*{2cm}}} \mathbb{K}, $$
the operator $Id_M \otimes (Q_L^*)^{-1}$ is continuous because $\alpha_{X;Y}$ is uniform (Proposition \ref{porpr}), and the functional $\varphi_{Q_L \circ T \circ I_M}$ is continuous with $\left\|\varphi_{Q_{L}\circ T \circ I_{M}} \right\|\le \left\| Q_{L}\circ T \circ I_{M}\right\|_{X;Y^{\rm dual}}  $ by Lemma \ref{1lema.ideal.maximal.qn} because $Q_{L}\circ T \circ I_{M}$ belongs to  $\mathcal{L}_{X;Y^{\rm dual}}(M;F/L)$. It follows that
\begin{align*} |\varphi_T(u)| & = \left|\sum_{j=1}^n y_j^* ( T(x_j))  \right| = \left|\sum_{j=1}^n Q_L^*(\psi_j)(T(x_j))\right| \\
&=  \left|\sum_{j=1}^n \psi_j(Q_L(T(x_j))\right| =  \left|\sum_{j=1}^n  \varphi_{Q_L \circ T \circ I_M}\left(x_j \otimes \psi_j \right)\right| \\
& = \left|\sum_{j=1}^n [ \varphi_{Q_L \circ T \circ I_M} \circ (Id_M \otimes (Q_L^*)^{-1})]\left( x_j \otimes y_j^* \right)\right|\\
&=  \left| [ \varphi_{Q_L \circ T \circ I_M} \circ (Id_M \otimes (Q_L^*)^{-1})]\left( \sum_{j=1}^n x_j \otimes y_j^* \right)\right| \\
&= \left| [ \varphi_{Q_L \circ T \circ I_M} \circ (Id_M \otimes (Q_L^*)^{-1})](u)\right|\\
&\leq \|\varphi_{Q_L \circ T \circ I_M}\|\cdot \|Id_M \otimes (Q_L^*)^{-1}) \| \cdot \alpha_{X,Y}(u ; M \otimes N)\\
& \leq \left\| Q_{L}\circ T \circ I_{M}\right\|_{X;Y^{\rm dual}}  \cdot \|Id_M \|\cdot \|(Q_L^*)^{-1}) \| \cdot (1+\eta)\alpha_{X,Y}(u ; E \otimes F^*)\\
& \leq s \cdot (1+\eta)\alpha_{X,Y}(u ; E \otimes F^*).
\end{align*}
Making $\eta \longrightarrow 0^+$ we get $\left|\varphi_{T}(u) \right|\leq s\cdot \alpha_{X,Y}^{}\left(u \right)$, which implies the continuity of the functional $\varphi_{T}\colon E\otimes_{\alpha_{X,Y}^{}}F^*\longrightarrow \mathbb{K}$ and $\|\varphi_T\| \leq s$. Calling Lemma \ref{1lema.ideal.maximal.qn} once again it follows that $T\in \mathcal{L}_{X;Y^{\rm dual}}(E;F)$ and $\left\|T \right\|_{X;Y^{\rm dual}}\le \left\|\varphi_{T} \right\|\leq s .$ This completes the proof of the first assertion.

The second assertion follows from the first one combined with \cite[8.11]{R. Ryan} (see also \cite[Theorem 8.7.5]{A.Pietsch}).
\end{proof}

The next corollary is just a combination of the theorem above with \cite[Corollary 17.8(4)]{A.Defant}.

\begin{corollary} Let $u \in {\cal L}(E;F)$ be given. Under the assumptions of Theorem \ref{maxim} we have $u \in \mathcal{L}_{X;Y^{\rm dual}}(E;F)$ if and only if $u^{**} \in \mathcal{L}_{X;Y^{\rm dual}}(E^{**};F^{**})$ and $\|u\|_{X,Y^{\rm dual}} = \|u^{**}\|_{X,Y^{\rm dual}}$.
\end{corollary}

\begin{examples}\rm (a) Theorem \ref{maxim} recovers the following well known facts.\\
$\bullet$ The Banach ideal of absolutely $(q,p)$-summing operators:
$$\Pi_{q,p} := {\cal L}_{\ell_p^w(\cdot); \ell_q(\cdot)} = {\cal L}_{\ell_p^w(\cdot); [\ell_{q^*}(\cdot)]^{\rm dual}},$$
$1 \leq p \leq q < \infty$, is maximal \cite[Proposition 10.2]{J.Diestel}. In particular, the ideal $\Pi_p$ of absolutely $p$-summing operators is maximal.\\
$\bullet$ The Banach ideal of Cohen strongly $(q,p)$-summing operators:
$${\cal D}_{q,p} := {\cal L}_{\ell_p(\cdot); \ell_q\langle\cdot \rangle} = {\cal L}_{\ell_p(\cdot); [\ell^w_{q*}(\cdot )]^{\rm dual}},$$
$1 \leq p \leq q < \infty$, is maximal. Although we found no reference to quote, we believe this is a well known fact. \\
$\bullet$ The Banach ideal of cotype $q$ operators:
$$\mathfrak{C}_q := {\cal L}_{{\rm RAD}(\cdot); \ell_q(\cdot)} = {\cal L}_{{\rm RAD}(\cdot); [\ell_{q^*}(\cdot)]^{\rm dual}},$$
$2 \leq q < \infty$, is maximal \cite[17.4]{A.Defant}.\\
{\rm (b)} Just to illustrate the new information that can be obtained from  Theorem \ref{maxim} we mention that the Banach ideals
$${\cal L}_{\ell_p^{\rm mid}(\cdot); \ell_q\langle\cdot \rangle} = {\cal L}_{\ell_p^{\rm mid}(\cdot); [\ell^w_{q*}(\cdot )]^{\rm dual}} {\rm ~~and~~} {\cal L}_{\ell_p^{\rm mid}(\cdot); \ell_q (\cdot )} = {\cal L}_{\ell_p^{\rm mid}(\cdot); [\ell_{q*}(\cdot )]^{\rm dual}},$$
which were studied in \cite{espaco.mid, J.R.Campos.J.Santos}, are maximal.
\end{examples}

\section{The dual of $E \otimes_{\alpha_{X,Y}}F$}

As is usual in the case of tensor norms (see \cite{A.Defant, R. Ryan}), for the tensor quasi-norm $\alpha_{X,Y}$ we describe the linear functionals on $E \otimes_{\alpha_{X,Y}}F $ as linear operators from $E$ to $F^*$ and as bilinear forms on $E \times F$. As a consequence we show when these tensor quasi-norms satisfy a condition that is equivalent to maximality of the corresponding operator ideal in the case of tensor norms.

For the first part of this section, which comprises Theorems \ref{dualB} and \ref{Teo.prin.1}, we want $\alpha_{X,Y}$ to be a reasonable quasi-norm, so we will suppose that $X$ and $Y$ are monotone sequences classes and $\varepsilon \leq \alpha_{X,Y}$.

To describe the linear functionals on $E \otimes_{\alpha_{X,Y}}F $ as bilinear forms we need one more concept introduced in \cite{G. Botelho}. A bilinear form $A \colon E \times F \longrightarrow \mathbb{K}$ is said to be $(X,Y;\ell_1)$-summing if $(A(x_j,y_j))_{j=1}^\infty \in \ell_1$ whenever $(x_j)_{j=1}^\infty \in X(E)$ and $(y_j)_{j=1}^\infty \in Y(F)$. The space ${\cal L}_{X,Y;\ell_1}(E,F;\mathbb{K})$ of all such bilinear forms is a Banach space under the norm
$$\|A\|_{X,Y;\ell_1} = \sup\left\{\| (A(x_j,y_j))_{j=1}^\infty \|_1  : (x_j)_{j=1}^\infty \in B_{X(E)},  (y_j)_{j=1}^\infty \in B_{Y(F)}\right\}.$$

We also need a property that is neither weaker nor stronger than being spherically complete: a sequence class $X$ is said to be {\it finitely boundedly complete}, {\it FBC} for short, if for any Banach space $E$, every $n \in \mathbb{N}$ and all   $(x_{j})_{j=1}^{n}\in X(E)$  and $(\lambda_{j})_{j=1}^{n}\in \ell_{\infty}$, it holds $(\lambda_{j}x_{j})_{j=1}^{n}\in X(E)$ and
	$\left\|(\lambda_{j}x_{j})_{j=1}^{n} \right\|_{X(E)}\leq \left\|(\lambda_{j})_{j=1}^{n} \right\|_\infty\cdot \left\|(x_{j})_{j=1}^{n} \right\|_{X(E)}. $

It is easy to check that the sequence classes $c_0(\cdot), \ell_\infty(\cdot), \ell_p(\cdot), \ell_p^w(\cdot), \ell_p^u(\cdot), \ell_p\langle \cdot \rangle, \ell_p^{\rm mid}(\cdot)$, $1 \leq p < \infty$, are FBC. Kahane's contraction principle \cite[12.2]{J.Diestel} guarantees that ${\rm Rad}(\cdot)$ and ${\rm RAD}(\cdot)$ are FBC in the real case $\mathbb{K} = \mathbb{R}$.

\begin{theorem}\label{dualB}
	Suppose that $X$ and $Y$ are finitely determined or finitely dominated and that one of them is FBC. Then,
	\[(E\otimes_{\alpha_{X,Y}^{}}F)^*\stackrel{1}{=} \mathcal{L}_{X,Y;\ell_{1}}(E,F;\mathbb{K}).\]
\end{theorem}

\begin{proof}
Let us see that map
	$\Psi \colon \mathcal{L}_{X,Y;\ell_{1}}(E,F;\mathbb{K})\longrightarrow (E\otimes_{\alpha_{X,Y}^{}}F)^*$ given by
	$$\Psi(A)\left(\sum_{j=1}^{n}x_{j}\otimes y_{j} \right)= \sum_{j=1}^{n}A(x_{j},y_{j}), $$
is a well defined bounded linear operator. First note that $\Psi(A)$ is the linearization of the bilinear form $A$, so it is a well defined linear functional on $E \otimes F$. To check its continuity with respect to $\alpha_{X,Y}$, note that,
for $u=\sum\limits_{j=1}^{n}x_{j}\otimes y_{j}\in E\otimes F$, since $A\in \mathcal{L}_{X,Y;\ell_{1}}(E,F;\mathbb{K})$  we have
	\begin{align*}
	\left|\Psi(A)(u) \right| = \left|\sum_{j=1}^{n}A(x_{j},y_{j}) \right|\leq \sum_{j=1}^{n}\left|A(x_{j},y_{j}) \right|  \leq \left\|A \right\|_{X,Y;\ell_{1}}\cdot \left\|(x_{j})_{j=1}^{n} \right\|_{X(E)}\cdot \left\|(y_{j})_{j=1}^{n} \right\|_{Y(F)}.
	\end{align*}
Taking the infimum over all representations of $u$ it follows that $\left|\Psi(A)(u) \right|\leq \left\|A \right\|_{X,Y;\ell_{1}}\alpha_{X,Y}^{}(u)$, proving that $\Psi(A)$ is a bounded linear functional and
$\left\|\Psi(A) \right\|\leq \left\|A \right\|_{X,Y;\ell_{1}}$.
	
Now it is enough to prove that $\Psi$ is a surjective isometry. Given $\varphi\in (E\otimes_{\alpha_{X,Y}^{}}F)^*$, it is clear that
$$A_{\varphi} \colon E\times F\longrightarrow \mathbb{K} : A_{\varphi}(x,y)= \varphi(x\otimes y),$$
is a bilinear form. Taking $(x_{j})_{j=1}^{n}$ in $E$ and $(y_{j})_{j=1}^{n}$ in $F$, assuming $X$ is FBC (the other case is analogous), we get
	\begin{align*}
	\left\|(A_{\varphi}(x_{j},y_{j}))_{j=1}^{n} \right\|_{1} &=  \sup_{(\lambda_{j})_{j=1}^{\infty}\in B_{\ell_{\infty}}}\left|\sum_{j=1}^{n}\lambda_{j} A_{\varphi}(x_{j},y_{j}) \right|
	= \sup_{(\lambda_{j})_{j=1}^{\infty}\in B_{\ell_{\infty}}}\left|\varphi\left( \sum_{j=1}^{n}(\lambda_{j}x_{j})\otimes y_{j})\right) \right|\\
	&\leq \sup_{(\lambda_{j})_{j=1}^{\infty}\in B_{\ell_{\infty}}}\left\|\varphi \right\|\cdot \left\|(\lambda_{j}x_{j})_{j=1}^{n} \right\|_{X(E)}\cdot \left\|(y_{j})_{j=1}^{n} \right\|_{Y(F)}\\
	&\leq \left\|\varphi \right\|\cdot \left\|(x_{j})_{j=1}^{n} \right\|_{X(E)}\cdot \left\|(y_{j})_{j=1}^{n} \right\|_{Y(F)}.
	\end{align*}
Using that $X$ and $Y$ are finitely determined we conclude that $A_{\varphi}\in \mathcal{L}_{X,Y;\ell_{1}}(E,F;\mathbb{K})$ and
	$\left\|A_{\varphi} \right\|_{X,Y;\ell_{1}}\leq \left\| \varphi \right\|$. A straightforward computation shows that $\Psi(A_{\varphi})= \varphi$ and completes the proof.
\end{proof}


Now we represent linear functionals on $E\otimes_{\alpha_{X,Y}^{}}F$ as linear operators from $E$ to $F^*$.

\begin{theorem}\label{Teo.prin.1}
	 Suppose that $X$ and $Y$ are finitely determined and $Y$ is spherically complete. Then, \[(E\otimes_{\alpha_{X,Y}^{}} F )^*\stackrel{1}{=} \mathcal{L}_{X;Y^{\rm dual}}(E;F^*).\]	
\end{theorem}

\begin{proof} Given $T \in \mathcal{L}_{X;Y^{\rm dual}}(E;F^*)$, call $\Psi(T)$ the linearization of the bilinear form
$$(x,y) \in E \times F \mapsto T(x)(y) \in \mathbb{K}. $$
So, $\Psi(T)$ is a linear functional on $E \otimes F$ such that
\[\Psi\left( T\right) \left( \sum_{j=1}^{n}x_{j}\otimes y_{j}\right) = \sum_{j=1}^{n}T\left( x_{j}\right) \left( y_{j}\right).\]
To check its continuity with respect to $\alpha_{X,Y}$, note that for $u=\sum\limits_{j=1}^{n}x_{j}\otimes y_{j}\in E\otimes F$, denoting by $J_F \colon F \longrightarrow F^{**}$ the canonical embedding,
	\begin{align*}
	\left|\Psi(T)(u) \right|&\leq \sum_{j=1}^{n}\left|T(x_{j})(y_{j}) \right|= \sum_{j=1}^{n} \left| J_{F}(y_{j})(T(x_{j})) \right|\\
& \leq  \left\|T \right\|_{X;Y^{\rm dual}} \cdot \left\|(x_{j})_{j=1}^{n} \right\|_{X(E)}\cdot \left\|(J_F(y_{j}))_{j=1}^{n} \right\|_{Y(F^{**})}\\
	&\leq \left\|T \right\|_{X;Y^{\rm dual}} \cdot \left\|(x_{j})_{j=1}^{n} \right\|_{X(E)}\cdot \left\|(y_{j})_{j=1}^{n} \right\|_{Y(F)},
	\end{align*}
where the last inequality follows from the linear stability of $Y$. Since this estimate holds for every representation of $u$, it follows that $\left|\Psi(T)(u) \right|\leq \left\|T \right\|_{X;Y^{\rm dual}}\cdot \alpha_{X,Y}^{}(u) $. Hence, $\Psi(T) \in (E\otimes_{\alpha_{X,Y}^{}} F )^*$, proving that $\Psi \colon \mathcal{L}_{X;Y^{\rm dual}}(E;F^*)\longrightarrow (E\otimes_{\alpha_{X,Y}^{}} F )^*$ is a (obviously linear) bounded operator with $\left\|\Psi(T) \right\|\leq \left\|T \right\|_{X;Y^{\rm dual}}.  $ Recall that $\mathcal{L}_{X;Y^{\rm dual}}(E;F^*)$ is a Banach space because $Y$ is spherically complete.
	
Now it is enough to show that $\Psi$ is a surjective isometry. For $\varphi\in (E\otimes_{\alpha_{X,Y}^{}}F)^*$, the map
$$T_{\varphi}\colon E\longrightarrow F^*~,~T_{\varphi}(x)(y)= \varphi(x\otimes y),$$
is clearly a bounded linear operator. Given  $(x_{j})_{j=1}^{\infty} \in X(E)$, for every $n \in \mathbb{N}$,
	\begin{align*}
	\left\|\left( T_{\varphi}(x_{j})\right)_{j=1}^{n}  \right\|_{Y^{\rm dual}(F^*)} = \sup_{(y_{j})_{j=1}^{\infty}\in B_{Y(F)}}\left|\sum_{j=1}^{n}\varphi(x_{j}\otimes y_{j})\right|
	\leq \left\|\varphi \cdot \right\|\left\|(x_{j})_{j=1}^{n} \right\|_{X(E)}.
	\end{align*}
 Since $X$ and $Y^{\rm dual}$ are finitely determined, taking the supremum over $n$ it follows that $T_{\varphi}\in \mathcal{L}_{X;Y^{\rm dual}}(E;F^*)$ and
	$\left\|T_{\varphi} \right\|_{X;Y^{\rm dual}} \leq \left\|\varphi \right\|$. The easily checked equality $\Psi(T_{\varphi})=\varphi$ completes the proof. 	
\end{proof}

It is well known (see, e.g., \cite[Ex. 17.2]{A.Defant}) that  a normed operator ideal $\cal I$ is maximal if and only if there exists a finitely generated tensor norm $\alpha$ such that
\begin{equation}\label{99999}
{\cal I}(E;F)\stackrel{1}{=} (E \otimes_{\alpha}F^*)^*\cap \mathcal{L}(E,F)
\end{equation}
for all Banach $E$ and $F$. We finish the paper establishing conditions under which the tensor quasi-norm $\alpha_{X,Y}^{}$ satisfies \eqref{99999}. Therefore, according to the Propositions \ref{razoavel} and \ref{porpr}, we will assume that $X$ and $Y$ are linearly stable, monotone and finitely injective sequences classes and $\varepsilon \leq \alpha_{X,Y}$.

\begin{theorem} Suppose that $X$ and $Y$ are  finitely determined or finitely  dominated and that $Y$ is  spherically complete. Then, regardless of the Banach spaces $E$ and $F$,
	$$\mathcal{L}_{X;Y^{\rm dual}}(E;F)\stackrel{1}{=}(E\otimes_{\alpha_{X,Y}^{}}F^*)^*\cap \mathcal{L}(E;F).$$
\end{theorem}

\begin{proof} The assumptions guarantee that $\mathcal{L}_{X;Y^{\rm dual}}(E;F)$ is a Banach space. The same arguments used before show that $\Phi \colon \mathcal{L}_{X;Y^{\rm dual}}(E,F)\longrightarrow (E\otimes_{\alpha_{X,Y}}F^*)^*\cap \mathcal{L}(E,F)$
	given by
	\[\Phi(T)\left(\sum_{j=1}^{n}x_{j}\otimes \psi_{j} \right) = \sum_{j=1}^{n} \psi_{j}(T(x_{j})),\]
is a well defined linear operator. For $u=\sum\limits_{j=1}^{n}x_{j}\otimes \psi_{j}\in E\otimes F^*$,
	\begin{align*}
	\left|\Phi(T)(u) \right| \leq \sum_{j=1}^{n}\left|\psi_{j}(T(x_{j})) \right|  \leq \left\|T \right\|_{X;Y^{\rm dual}}\cdot \left\|(x_{j})_{j=1}^{n} \right\|_{X(E)} \cdot \left\|(\psi_{j})_{j=1}^{n} \right\|_{Y(F^*)}.
	\end{align*}
Again, since this holds for any representation of $u$ we have  	$
	\left|\Phi(T)(u) \right|\leq \left\|T \right\|_{X;Y^{\rm dual}} \cdot \alpha_{X,Y}^{}(u)
	$, which proves, in particular, that $\Phi(T)$ is continuous and $\left\|\Phi(T) \right\|\leq \left\| T\right\|_{X;Y^{\rm dual}}.$
	
Once again, it is enough to show that $\Phi$ is a surjective isometry. To do so, given $\varphi\in (E\otimes_{\alpha_{X,Y}^{}}F^*)^*\cap  \mathcal{L}(E,F)$ consider the continuous linear operator $T_{\varphi} \colon E\longrightarrow J_{F}(F) \subseteq F^{**}$ given by $T_{\varphi}(x)(\psi)= \varphi(x\otimes \psi)$. 	Let $M\in \mathcal{F}(E)$ and $L\in \mathcal{CF}(F)$ be given. Considering the isometric isomorphism $Q_{L}^* \colon (F/L)^{*}\longrightarrow L^{\perp}$, Proposition \ref{Teo.prin.1} gives
	\begin{equation}\label{eq.iso.iso}
	\left(M\otimes_{\alpha_{X,Y}^{}}\left(F/L\right)^* \right)^* \stackrel{1}{=} \mathcal{L}_{X;Y^{\rm dual}}\left(M; \left(F/L \right)^{**}  \right) \stackrel{1}{=} \mathcal{L}_{X;Y^{\rm dual}}\left(M;F/L \right).
	\end{equation}
Considering the composition
$$M \otimes_{\alpha_{X;Y}}  (F/L)^* \stackrel{Id_M \otimes Q_L^*}{\xrightarrow{\hspace*{2cm}}} M \otimes_{\alpha_{X;Y}}L^\perp \stackrel{\varphi|_{M \otimes_{\alpha_{X;Y}}L^\perp }}{\xrightarrow{\hspace*{2cm}}} \mathbb{K}, $$
we have $\varphi|_{M \otimes_{\alpha_{X;Y}}L^\perp }\circ (Id_M \otimes Q_L^*) \in \left(M\otimes_{\alpha_{X,Y}^{}}\left(F/L\right)^* \right)^*$. By (\ref{eq.iso.iso}) there is a unique $T \in \mathcal{L}_{X;Y^{\rm dual}}\left(M;F/L \right)$ such that $\Psi(T) = \varphi|_{M \otimes_{\alpha_{X;Y}}L^\perp }\circ (Id_M \otimes Q_L^*)$, where $\Psi$ is the isomorphism constructed in the proof of Proposition \ref{Teo.prin.1}. 	For every tensor  $\sum\limits_{j=1}^{n}x_{j}\otimes \psi_{j}\in M\otimes L^{\perp}$, we have $\sum\limits_{j=1}^{n}x_{j}\otimes (Q^*_{L})^{-1}(\psi_{j})\in M\otimes \left( F/L\right)^* $, so
\begin{align*}\Psi(T)& \left(\sum_{j=1}^{n}x_{j}\otimes (Q_{L}^*)^{-1}(\psi_{j}) \right) =	\varphi|_{M \otimes_{\alpha_{X;Y}}L^\perp }\circ (Id_{M}\otimes Q_{L}^*)\left(\sum_{j=1}^{n}x_{j}\otimes (Q_{L}^*)^{-1}(\psi_{j}) \right)\\
 &= \varphi\left(\sum_{j=1}^{n}x_{j}\otimes \psi_{j} \right)=   \sum_{j=1}^{n} T_{\varphi}(x_{j})(\psi_{j})  = \sum_{j=1}^{n}\psi_{j}(T_{\varphi}\circ Id_{M})(x_{j})\\
  &= \sum_{j=1}^{n}(T_{\varphi}\circ Id_{M})^*(\psi_{j})(x_{j})= \sum_{j=1}^{n}(Id_{M}^*\circ T_{\varphi}^*)(\psi_{j})(x_{j})\\
	&= \sum_{j=1}^{n}(Id_{M}^*\circ T_{\varphi}^*\circ Q_{L}^*)\left((Q_{L}^*)^{-1}(\psi_{j})\right)(x_{j})= \sum_{j=1}^{n}(Q_{L}^*)^{-1}(\psi_{j})(Q_{L}\circ T_{\varphi}\circ Id_{M})(x_{j})\\
& = \Psi(Q_L \circ T_\varphi \circ Id_M)\left(\sum_{j=1}^{n}x_{j}\otimes (Q_{L}^*)^{-1}(\psi_{j})  \right).
\end{align*}
 The injectivity of $\Psi$ gives $T= Q_{L}\circ T_{\varphi}\circ I_{M}$, and the fact that it is an isometry yields
\begin{align*}
	\left\|Q_{L}\circ T_{\varphi}\circ I_{M} \right\|_{X;Y^{\rm dual}}& = \|\Psi^{-1}(T)\|= \left\|\varphi|_{M \otimes_{\alpha_{X;Y}}L^\perp }\circ (Id_M \otimes Q_L^*)\right\|\\
 &\leq \left\| \varphi\right\|\cdot \left\| I_{M}\otimes Q_{L}^*\right\| \leq  \left\| \varphi\right\|.
\end{align*}
It follows from Theorem \ref{maxim} that $T_{\varphi}\in \mathcal{L}_{X;Y^{\rm dual}}(E;F)$ and
$$\left\|T_{\varphi} \right\|_{X;Y^{\rm dual}}\leq \sup_{M,L}\left\|Q_{L}\circ T_{\varphi}\circ I_{M} \right\|_{X;Y^{\rm dual}} \leq \left\|\varphi \right\|.$$ Finally, it is not difficult to see that $\Phi(T_{\varphi})=\varphi$.
\end{proof}

Plenty of concrete cases for which the theorem above applies can be provided just bearing in mind all that was said about the sequence classes listed in Example \ref{exsec}.

\bigskip

\noindent Faculdade de Matem\'atica\\
Universidade Federal de Uberl\^andia\\
38.400-902 -- Uberl\^andia -- Brazil\\
e-mail: botelho@ufu.br\\

\noindent Departamento de Ci\^{e}ncias Exatas\\
Universidade Federal da Para\'iba\\
58.297-000 -- Rio Tinto -- Brazil\\
\hspace*{1,7cm} and

\noindent Departamento de Matem\'atica\\
Universidade Federal da Para\'iba\\
58.051-900 -- Jo\~ao Pessoa -- Brazil

\noindent e-mails: jamilson@dcx.ufpb.br \, and/or \, jamilsonrc@gmail.com\\

\noindent Departamento de Matem\'atica\\
Universidade Federal da Para\'iba\\
58.051-900 -- Jo\~ao Pessoa -- Brazil\\
e-mail: llucascarvalho23@yahoo.com.br


\begin{thebibliography}{99}\small

\vspace*{-0.4em}
\bibitem{achour2018} D. Achour, A. Alouani, P. Rueda, E. A. S\'anchez-P\'erez, {\it Tensor characterizations of summing
polynomials}, Mediterr. J. Math. {\bf 15} (2018), no. 3, Paper No. 127, 12 pp.

\vspace*{-0.4em}
\bibitem{D.Achour} D. Achour and M. T. Belaib, \emph{Tensor norms related to the space of Cohen p-nuclear multilinear mappings}, Ann. Funct. Anal. {\bf 2} (2011), no. 1, 128--138.


\vspace*{-0.4em}
\bibitem {complut} G. Botelho and J. R. Campos, {\it Type and cotype of multilinear operators}, Rev. Mat. Complut. {\bf 29} (2016), no. 3, 659–676.


\vspace*{-0.4em}
\bibitem {G. Botelho} G. Botelho and J. R. Campos, \emph{On the transformation of vector-valued sequences by linear and multilinear operators}, Monatsh. Math. {\bf 183} (2017), no. 3, 415--435.

\vspace*{-0.4em}
\bibitem{Jamilson.dual} G. Botelho and J. R. Campos, \emph{Duality theory for generalized summing linear operators}, Collect. Math. (2022). https://doi.org/10.1007/s13348-022-00359-4.

\vspace*{-0.4em}
\bibitem{espaco.mid} G. Botelho, J. R. Campos and J. Santos, \emph{Operator ideals related to absolutely summing and Cohen strongly summing operators}, Pacific J. Math. {\bf 287} (2017), no. 1, 1--17.


\vspace*{-0.4em}
\bibitem{G. Botelho and D. Freitas} G. Botelho and D. Freitas, \emph{Summing multilinear operators by blocks: The isotropic and anisotropic cases}, J. Math. Anal. Appl. {\bf 490} (2020), no. 1, 124203, 21pp.

\vspace*{-0.4em}
\bibitem{Botelhojlucas} G. Botelho and J. L. P. Luiz, \emph{Complete latticeability in vector-valued sequence spaces}, Math. Nachr., to appear.

\vspace*{-0.4em}
\bibitem{raquel} G. Botelho and R. Wood, {\it Hyper-ideals of multilinear operators and two-sided polynomial ideals generated by sequence classes}, Mediterr. J. Math., to appear (available at arXiv:2110.00051v1[mathFA], 2021.

\vspace*{-0.4em}
\bibitem{J.R.Campos.J.Santos} J. R. Campos and J. Santos, \emph{An anisotropic approach to mid summable sequences}, Colloq. Math., {\bf 161} (2020), no. 1, 35--49.

\vspace*{-0.4em}
\bibitem{argentinos} D. Carando, M. Mazzitelli, P. Sevilla-Peris, {\it A note on the symmetry of sequence spaces}, Math. Notes {\bf 110} (2021), no. 1-2, 26–40.

\vspace*{-0.4em}
\bibitem{chen} D. Chen, J. Alejandro Ch\'avez-Dom\'inguez, L. Li, {\it $p$-converging operators and Dunford-Pettis property of order $p$}, J. Math. Anal. Appl. {\bf 461} (2018), 1043--1066.

\vspace*{-0.4em}
\bibitem {sheldon} S. Dantas, M. Jung, O. Rold\'an, A. Rueda Zoca, {\it Norm-attaining tensors and nuclear operators}, Mediterr. J. Math.  {\bf 19} (2022), no. 1, Paper No. 38, 27 pp.

\vspace*{-0.4em}
\bibitem {A.Defant} A. Defant and K. Floret, \emph{Tensor norms and operator ideals}, Vol. {\bf 176}, North-Holland Mathematics Studies, 1993.

\vspace*{-0.4em}
\bibitem {handbook}
J. Diestel, H. Jarchow, A. Pietsch, {\it Operator ideals}, Handbook of the geometry of Banach spaces, Vol. I, 437–496, North-Holland, Amsterdam, 2001.

\vspace*{-0.4em}
\bibitem{J.Diestel} J. Diestel, H. Jarchow and A. Tonge, \emph{Absolutely Summing Operators}, Cambridge Studies in Advanced Mathematics, vol {\bf 43}, Cambridge University Press, 1995.

\vspace*{-0.4em}
\bibitem{maite} M. Fern\'andez-Unzueta, S. Garc\'ia-Hern\'andez, {\it $(p,q)$-dominated multilinear operators and Laprest\'e tensor norms}, J. Math. Anal. Appl. {\bf 470} (2019), no. 2, 982–1003.


\vspace*{-0.4em}
\bibitem{maite2020} M. Fern\'andez-Unzueta, L. F. Higueras-Monta\~no, {\it A general theory of tensor products of convex sets in Euclidean spaces}, Positivity {\bf 24} (2020), no. 5, 1373–1398.

\vspace*{-0.4em}
\bibitem{samuel} S. Garc\'ia-Hern\'andez, {\it The duality between ideals of multilinear operators and tensor norms}, Rev. R. Acad. Cienc. Exactas Fís. Nat. Ser. A Mat. RACSAM {\bf 114} (2020), no. 2, Paper No. 77, 25 pp.


\vspace*{-0.4em}
\bibitem{N. J. Kalton 2} N. J. Kalton, \emph{Quasi-Banach spaces}, Handbook of the geometry of Banach spaces, Vol. 2, 1099--1130, North-Holland, 2003.


\vspace*{-0.4em}
\bibitem{kim2020} J. M. Kim, {\it Approximation properties of tensor norms and operator ideals for Banach spaces},  Open Math. {\bf 18} (2020), no. 1, 1698–1708.


\vspace*{-0.4em}
\bibitem{kim} J. M. Kim, S. Lassalle, P. Turco, {\it Lifting some approximation properties from a dual space $X'$ to the Banach space $X$}, Studia Math. {\bf 257} (2021), 287--294.


\vspace*{-0.4em}
\bibitem{J. A. Lopez Molina} J. A. López Molina, \emph{$(n + 1)$-Tensor norms of Lapresté's type}, Glasgow Math. J., {\bf 54} (2012), 665--692.


\vspace*{-0.4em}
\bibitem{J. A. Lopez Molina.tres} J. A. López Molina, \emph{The minimal and maximal operator ideals associated to $(n + 1)$-tensor norms of Michor's type}, Positivity, {\bf 22} (2018), 1109--1142.

\vspace*{-0.4em}
\bibitem{lopezmolina2019}  J. A. López Molina, {\it  Multiple tensor norms of Michor's type and associated operator ideals}, Descriptive topology and functional analysis. II, 191–224, Springer Proc. Math. Stat., 286, Springer, Cham, 2019.

\vspace*{-0.4em}
\bibitem{miguel} M. Mart\'in, J. Mer\'i, A. Quero, {\it Numerical index and Daugavet property of operator ideals and tensor products}, Mediterr. J. Math. {\bf 18} (2021), no. 2, Paper No. 72, 15 pp.


\vspace*{-0.4em}
\bibitem{A.Pietsch} A. Pietsch, \emph{Operator Ideals,} North-Holland, 1980.


\vspace*{-0.4em}
\bibitem{history} A. Pietsch, {\it History of Banach spaces and linear operators}, Birkh\"auser Boston, Inc., Boston, MA, 2007.


\vspace*{-0.4em}
\bibitem {J. Ribeiro and F. Santos} J. Ribeiro and F. Santos, {\it Generalized multiple summing multilinear operators on Banach spaces}, Mediterr. J. Math., \textbf{16} (2019), Paper no. 108, 20 pp.

\vspace*{-0.4em}
\bibitem {J. Ribeiro and F. Santos.dois} J. Ribeiro and F. Santos, {\it Absolutely summing polynomials},  Methods Funct. Anal. Topology {\bf 27} (2021), no. 1, 74--87.


\vspace*{-0.4em}
\bibitem {R. Ryan} R. Ryan, \emph{Introduction to tensor product of Banach spaces}, Springer-Verlag, 2002.


\vspace*{-0.4em}
\bibitem {turcovillafane} P. Turco, R. Villafa\~ne, {\it Galois connection between Lipschitz and linear operator ideals and minimal Lipschitz operator ideals}, J. Funct. Anal. {\bf 277} (2019), 434--451.

\vspace*{-0.5em}
\bibitem{tarieladze} N. N. Vakhania, V. I. Tarieladze, S. A. Chobanyan, {\it Probability Distributions on Banach Spaces}, D.
Reidel Publishing Co., Dordrecht, 1987.

\end{thebibliography}
\end{document}